\documentclass[sn-mathphys,Numbered]{sn-jnl}

\usepackage{graphicx}%
\usepackage{multirow}%
\usepackage{amsmath,amssymb,amsfonts}%
\usepackage{amsthm}%
\usepackage{mathrsfs}%
\usepackage[title]{appendix}%
\usepackage{xcolor}%
\usepackage{textcomp}%
\usepackage{manyfoot}%
\usepackage{booktabs}%
\usepackage{algorithm}%
\usepackage{algorithmicx}%
\usepackage{algpseudocode}%
\usepackage{listings}%

\theoremstyle{thmstyleone}%
\newtheorem{theorem}{Theorem}
\newtheorem{proposition}[theorem]{Proposition}%

\theoremstyle{thmstyletwo}%

\theoremstyle{thmstylethree}%
\newtheorem{definition}{Definition}%

\begin{document}

\title[O.K.V.S.R.I]{On KV-Poisson Structure and related invariants.}


\author*[1]{\fnm{} \sur{Prosper Rosaire Mama  Assandje }}\email{mamarosaire@facsciences-uy1.cm}\equalcont{These authors
contributed equally to this work.}
\author[2]{\fnm{} \sur{ Herguey Mopeng}}\email{mopengpeter@gmail.com}
\author[3]{\fnm{} \sur{Joseph Dongho }}\email{josephdongho@yahoo.fr}
\equalcont{These authors contributed equally to this work.}

  \affil*[1]{\orgdiv{Department of Mathematics}, \orgname{University of Yaounde 1},
\orgaddress{\street{usrectorat@.univ-yaounde1.cm}, \city{yaounde},
\postcode{337}, \state{Center}, \country{Cameroon}}}

\affil[2]{\orgdiv{Department of Mathematics}, \orgname{University of
Douala}, \orgaddress{\street{rectorat@univ-douala.com},
\city{Douala}, \postcode{2701}, \state{Littoral},
\country{Cameroon}}}

\affil[3]{\orgdiv{Department of Mathematics and Computer Science},
\orgname{University of Maroua},
\orgaddress{\street{decanat@fs.univ-maroua.cm}, \city{Maroua},
\postcode{814}, \state{Far -North}, \country{Cameroon}}}


\abstract{We propose an deepened analysis of KV-Poisson structures
of on $\mathrm{I\! R}^{2}.$ We present their classification their
properties an their possible applications in different domains. We
prove that these structure give rise to a new Cohomological
invariant. We explicitly compute  the Cohomological groups of some
of these structures.}




\maketitle

\section{Introduction}\label{sec1}
The characterization of the KV-Poisson structures of $\mathrm{I\!
R}^{2}$ is a recent research topic in mathematics that is attracting
 increasing interest. KV-Poisson structures is Poisson structures associated to a KV-algebra structures.
They appear in many mathematical areas, in particular in symplectic
geometry,  field theory  and mathematical physics. The relations
between the KV-Cohomology of the Koszul-Vinberg algebras and some
properties of various geometric objects have been  studied in
\cite{BOYOM4}. More precisely, \cite{BOYOM4} proved  that the scalar
KV-Cohomology of the algebras of real Koszul-Vinbergs or holomorphic
is closely  related to the real Poisson manifolds. M.N. Boyom has
also underlined the close relations between the Nijenhuis's works
\cite {Koszul} and the KV-Cohomology. In \cite{BOYOM5}, M.N. Boyom
studied KV cohomology, which is a form of cohomology associated with
certain geometric structures, particularly in the context of the
flat affine manifold. M.N. Boyom investigated how the KV structures
apply to the flat affine manifold. He studied the topological and
geometric properties of these manifolds, as well as the applications
of KV cohomology to the study of their structures.
 In \cite{Landsman}, N.P.Landsman had explored the more advanced Poisson structures,
  including the structures of Kozul-Vinberg-Poisson, which are important generalizations
   of the standard Poisson structure. He also approached the applications of Poisson structures
    in mathematical physics, emphasizing their role in classical mechanics, quantum field theory,
     sympletic geometry and other related areas. In the recent works of \cite{Mopeng} a complete
      classification of KV-structures on the real affine plane is proposed. The authors there also
       determine all the KV-hessian. Having on the affine plan an algebraic structure, knowing
       the usefulness of the Poisson structures as well in the classical mechanics as in the quantum
        mechanics, we wonder if there exist KV-structures that are
  Poisson structures on the real affine plan? What are therefore the possible
  different classifications of the structures of KV-Poisson of $ \mathrm{I\! R}^{2}$?
 How can we characterize these structures in terms
  of their algebraic and geometric properties? In this
   paper we give some answers to these questions by determining
    all KV-structures that are Poisson structures; then we propose
     a deeper study of the different KV-Poisson structures on the
      vectorial space $ \mathrm{I\! R}^{2} $. This article begins in
      section \ref{sec1}, with an introduction, in section \ref{sec2}, the
preliminaries on the basic concepts of KV algebras and Poisson
algebras, It also presents the different classes of KV-Poisson
structures, with emphasis on their geometric
        and algebraic properties. In section \ref{sec3}, we present the characterization of the KV-Poisson structures of $\mathrm{I\! R}^{2}$;
        in section \ref{sec4},the complex of KV-Poisson Cohomology; in the
   ends in section \ref{sec5},  the conclusion.
\section{Preliminaries}\label{sec2}
In this first part, we introduce the main notions, notably some
definitions, bound to our study. Most of these results are proved in
\cite{BOYOM1},\cite{BOYOM2},\cite{BOYOM3}, \cite{Dito} and
\cite{Boucetta}. One recovers also some in \cite{Bourbaki},
\cite{Landsman} and \cite{DRUEL}.

\subsection{Koszul-Vinberg algebra}
$K$ denote a commutative field of characteristic zero. Let $A$ be a
$K$-algebra; $a, b$ and $c$ elements of $A.$ We consider on $A$ the
following operations.
\begin{enumerate}
\item[(i)] A multiplication on $A.$ It is an K-bilinear map defined by:
\begin{eqnarray*}
\mu :(a,b)\mapsto ab \in A
\end{eqnarray*}
\item[(ii)] The associator of $\mu.$ It is the trilinear map defined
by:
\begin{eqnarray*}
Ass_{\mu}: (a,b,c) \mapsto \mu(\mu(a,b),c)-\mu(a,\mu(b,c))\in A.
\end{eqnarray*}
\item[(iii)] The KV-anomaly of $\mu.$ It is the trilinear map defined by:
\begin{eqnarray}
KV_{\mu}:(a,b,c) \mapsto Ass_{\mu}(a,b,c)-Ass_{\mu}(b,a,c)\in A.
\end{eqnarray}
In other words,
$$KV_\mu(a,b,c)=\mu(\mu(a,b),c)+\mu(b,\mu(a,c))-\mu(a,\mu(b,c))-\mu(\mu(b,a),c)$$
\end{enumerate}
\begin{definition}\cite{BOYOM3}
A K-algebra $(A, \mu)$ is called a KV-algebra or Koszul-Vinberg
algebra, if the associator of it structure $\mu$ is symmetric with
respect to the first two variables. That is:  $KV_{\mu}=0$
\end{definition}
\subsubsection{KV-bimodule}
Let $A$ be a KV-algebra over $K$. Let $M$ be an vector space on $K$.
For any $a$ element of $A$ and $x$ element of $M,$ we define the
following two $K$-bilinear maps
\begin{center}
$\mu_{1}:(a,x)\mapsto ax \in M$ and $\mu_{2}:(x,a)\mapsto xa \in M$
\end{center}
For all  $b$ element of $A,$ we denote:
\begin{equation*}
Ass(a,b,x)=(ab)x-a(bx)=\mu_1(\mu(a,b), x)-\mu_1(a, \mu_1(b,x))
\end{equation*}
\begin{equation*}
Ass(a,x,b)=(ax)b-a(xb)=\mu_2(\mu_1(a,x),b)-\mu_1(a,\mu_2(x,b))
\end{equation*}
\begin{equation*}
Ass(x,a,b)=(xa)b-x(ab)=\mu_2(\mu_2(x,a),b)-\mu_2(a,\mu(a,b))
\end{equation*}
\begin{equation*}
Ass(b,a,x)=(ba)x-b(ax)=\mu_1(\mu(b,a),x)-\mu_1(b,\mu_1(a,x))
\end{equation*}
\begin{equation*}
KV_{\mu_1}(a,b,x)=Ass(a,b,x)-Ass(b,a,x)
\end{equation*}
\begin{equation*}
KV_{\mu_2}(a,x,b)=Ass(a,x,b)-Ass(x,a,b)
\end{equation*}
 \begin{definition}\cite{BOYOM3}
The  $K$-vector space $M$ provided with the operations $\mu_{1}$ and
$\mu_{2}$ is called a KV-bimodule on  $A$ if we have the following
properties:
\begin{enumerate}
\item[(i)] $KV_{\mu_{1}}(a,b,x)=0$,\\ i.e; $\mu_1(\mu(a, b), x)+\mu_1(b, \mu_1(a,x))=\mu_1(\mu(b,a),x)+\mu_1(a,\mu_1(b,x))$
\item[(ii)] $KV_{\mu_{2}}(a,x,b)=0$\\
i.e;
$\mu_2(\mu_1(a,x),b)+\mu_2(x,\mu(a,b))-\mu_1(a,\mu_2(x,b))-\mu_2(\mu_2(x,a),b)$
\end{enumerate}
We say that $M$ is a left (resp. right) KV-module if the bilinear
map $\mu_{2}$ (resp. $\mu_{1}$) is is hopeless.
\end{definition}
\subsubsection{KV-Cohomology}
Let $q$ be a positif integer. Let $A$ be a KV-algebra and $M$ be
KV-bimodule on $A$. Denote  $C^{q}(A,M)$ the $K$-vector space of
$q$-multilinear maps from $A$ to $M.$ For any element $a$ of $A$ and
for any  $f\in C^{q}(A,M)$, we define the following bilinear maps:
$(a,f)\mapsto a.f \in C^{q}(A,M)$ such that:
$(a_{1},...,a_{q})\mapsto a(f(a_{1},
...,a_{q}))-\sum_{j=1}^{q}f(a_{1},...,a_{j},...,a_{q})$ and
$(f,a)\mapsto f.a \in C^{q}(A,M)$ such that
$(a_{1},...,a_{q})\mapsto(f(a_{1}, ...,a_{q}))a.$ For any $(a_{1},
...,a_{q})\in A^{q}$. With  respect to these
maps, $C^{q}(A,M)$ become a KV-bimodule on $A.$ \\
For any $\rho\in\{1,...,q\}$, and $f\in C^{q}(A,M)$, we define the
following $A-$linear map
 $e_{\rho}(a)$, by
$e_{\rho}(a):f \mapsto e_{\rho}(a)f \in C^{q-1}(A,M);$ such that
$(a_{1}, ...,a_{q-1})\mapsto(e_{\rho}(a)f)(a_{1},
...,a_{q-1}):=f(a_{1}, ...,a_{\rho-1},a_{\rho},...,a_{q-1}),$ These
maps allow us to define the following operator: $\delta^{q}:f\mapsto
\delta^{q} f \in C^{q+1}(A,M)$
 such that
\begin{eqnarray}\label{eq2}
\delta^{q} f(a_{1}
...a_{q+1})=\underset{j=1}{\overset{q}\sum}(-1)^{j}\left\{(a_{j}f)(a_{1},...\widehat{a_{j}}...a_{q+1})
+(e_{q}(a_{j})(fa_{q+1}))(a_{1}...\widehat{a_{j}}...\widehat{a_{q+1}})\right\},
\end{eqnarray}
It was prove in \cite{BOYOM1} that $\delta^{q} \circ \delta^{q-1} =
0.$ Therefore, $Im \delta^{q-1}\subset Ker \delta^{q}.$

Set $C_{KV}(A,M)=\oplus_{q\geq 0}C^{q}(A,M),$ the coboundary
operator $\delta$ given by \ref{eq2} endows $C_{KV}(A,M)$ with the
structure of a graded cochain complex:
\begin{eqnarray*}
 ...\longrightarrow C^{q}(A,M)\stackrel{ \delta^{q} }{\longrightarrow}C^{q+1}(A,M)\stackrel{ \delta^{q+1} }{\longrightarrow}C^{q+2}(A,M)\longrightarrow ...
\end{eqnarray*}
\begin{definition}\cite{BOYOM1}, \cite{BOYOM3}
We call a KV-cohomology of the KV-algebra $A$ with coefficients in
the KV-bimodule $M$, the cohomology of the KV-complex
$(C_{KV}(A,M),\delta).$ It is denoted
 \begin{center}
 $ H_{KV}(A,M)=\oplus_{q \geq 0}H_{KV}^{q}(A,M)$
 \end{center}
\end{definition}
 where
\begin{equation}
 H_{KV}^{q}(A,M)=\frac{Ker \delta^{q}}{Im \delta^{q-1}}
\end{equation}
 $\bullet$ Elements of $Ker \delta^{q}$ are called $q-$cocycle of the complex $(C_{KV}(A,M),\delta)$;\\
 $\bullet$ Elements $Im \delta^{q-1}$ are called $q$-coboundary of complex $(C_{KV}(A,M),\delta)$;\\
 $\bullet$ The set $H_{KV}^{q}(A,M)$ is called the $q^{th}$ cohomological group of the complex $(C_{KV}(A,M),\delta)$.
\subsection{Poisson algebra}
\begin{definition}\cite{Boucetta},\cite{DRUEL}
Let $A$ be an associative and commutative algebra over a field $K$.
A Poisson bracket on $A$ is a $K-$ bilinear map $\{,\}$ on $A$ such
that:
\begin{enumerate}
\item[(i)] Skew-symmetry : $\{a,b\}= - \{b,a\}$;
\item[(ii)] Jacobi identity: $\{a,\{b,c\}\} + \{b,\{c,a\}\}+\{c,\{a,b\}\}=0$,
\item[(iii)]Leibniz rule: $\{\{a,b\},c\}=\{a,\{b,c\}\}+\{b,\{a,c\}\}$;\\
for all $a, b, c \in A.$
\end{enumerate}
\end{definition}
A differential manifold $M$ is called Poisson manifold if its
algebra $C^{\infty}(M)$ of smooth functions on $M$ is provided with a Poisson bracket. In that case, we denote $(M,\{,\})$ \\
Let $(M,\{,\})$ be a Poisson manifold; it follow from relations (i)
et (iii), that it exist an unique contravariant 2-tensor $\pi$, such
that: $\{f,g\}=\pi(df,dg)$, for all $f, g\in C^{\infty}(M)$. In
general, $\pi$ is called Poisson tensor or Poisson bivector of the
Poisson manifold $(M,\{,\})$.\\
Any Poisson Algebra $(A\{,\})$ induce an $A-$module homomorphism
$H:\Omega_A\rightarrow Der_A: da\mapsto H(da)b=\{a, b\}$ where
$\Omega_A$ is the $A$-module of K$\ddot{a}$lher differential of $A$
and $Der_A$ is the $A$-module of derivations of $A.$ For any $a\in
A$ we denote $H_a$ the element of $Der_A$ such that $H_a=H(da).$
$H_a$ is called hamiltonian derivation associated to  $a.$

\subsubsection{Poisson Cohomology}
Let  $(A,\{.,.\})$ be a Poisson algebra. For any positif integer
$q,$ we denote by $\mathfrak{X}^q(A)$ the $A$-module of poly
derivation skew symmetric of $A.$ By convention, we denote
$\mathfrak{X}^0(A)=A$. For any $Q\in \mathfrak{X}^q(A)$, we define
the following map
\begin{center}
 $ \delta^{q}:Q \mapsto \delta^{q} (Q ) \in \mathfrak{X}^{q+1}(A)$
\end{center}
such that
\begin{eqnarray}
\delta (Q)(F_{0}, ..., F_{q})&=&\sum_{i=0}^{q}(-1)^{i}\{F_{i},Q(F_{0},...,\widehat{F_{i}},...,F_{q})\} \nonumber\\
                             &+& \sum_{0\leq i<j\leq q}(-1)^{i+j}Q(\{F_{i},F_{j}\}, F_{0}, ...,\widehat{F_{i}},...,\widehat{F_{j}},...,F_{q}),
\end{eqnarray}
where $F_{0},...,F_{q}$ are elements of $A$, $q\geq 0.$ This map
verify $\delta^{q-1}\circ \delta^{q} = 0.$ Therefore $Im
\delta^{q-1}\subset Ker \delta^{q}.$
 Set $\mathfrak{X}(A)=\underset{q=0}{\overset{\infty}\bigoplus}\mathfrak{X}^q(A)$; we obtain the following complex:
\begin{eqnarray*}
...\longrightarrow \mathfrak{X}^q(A) \stackrel{ \delta^{q}
}{\longrightarrow}\mathfrak{X}^{q+1}(A)\stackrel{ \delta^{q+1}
}{\longrightarrow}\mathfrak{X}^{q+2}(A)\longrightarrow ...
\end{eqnarray*}
\begin{definition}\cite{Boucetta}, \cite{DRUEL}
The group
 $HP(A)=\oplus_{q \geq 0}HP^{q}(A)$ where
$ HP^{q}(A)=\frac{Ker \delta^{q}}{Im \delta^{q-1}} $ is called
Poisson cohomology group of the Poisson algebra $(A, \{-,-\}).$ The
set $HP^{q}(A)$ is called the $q^{th}$ Poisson Cohomological group
of $(A, \{-,-\}).$
\end{definition}
\section{Characterization of the KV-Poisson structures of $\mathrm{I\! R}^{2}$}\label{sec3}
In this section, we try to characterize KV-Poisson structures on the $\mathrm{I\! R}$ vector space of dimension $2$.\\
Let $(A, \mu)$ be a  KV-algebra. From the Jacobi identity:
$\mu(u,\mu(v,w))+\mu(v,\mu(w,u))+\mu(w,\mu(u,v))=0,$ we deduce that
if $\mu$ is a KV-Poisson structure, then: for all $u, v, w \in A$,
$\mu(w,\mu(u,v))=0$. This condition is sufficient when $\mu$ is skew
symmetric. The following proposition allows us to define the notion
of KV-Poisson structure on a given  KV-algebra $A$.
\begin{proposition}
A KV-structure $\mu$ on a KV-algebra $A$ is a Poisson  structure if
and only if:
\begin{enumerate}
\item[(i)] $\mu$ is skew symmetric;
\item[(ii)] For all $u, v, w \in A,$
$\mu(w,\mu(u,v))=0$.
\end{enumerate}
\end{proposition}
\begin{proof}It is obvious that if $\mu$ is a KV-Poisson structure,
then properties (i) and (ii) are satisfy. Since $ \mu$ is a
KV-structure, for all $u,v, w \in A,$ we have
\begin{eqnarray}\label{eq1}
\mu(\mu(u,v),w)-\mu(u,\mu(v,w))-\mu(\mu(v,u),w)+\mu(v,\mu(u,w))=0
\end{eqnarray}
If $\mu$ is more skew symmetric, equation (\ref{eq1}) become:
\begin{eqnarray*}
  -[\mu(u,\mu(v,w))+\mu(v,\mu(w,u))+\mu(w,\mu(u,v))]- \mu(w,\mu(u,v))=0
\end{eqnarray*}
What implies therefore that: $ \mu(w,\mu(u,v))=0$. And also, it
follows from the Leibniz rule that: $\mu(\mu(u,v),w)=
\mu(u,\mu(v,w))+\mu(v,\mu(u,w)),$ therefore,
$\mu(u,\mu(v,w))+\mu(v,\mu(u,w))=0$. Then, from the Jacobi identity,
one gets : $-\mu(v,\mu(w,v))+\mu(v,\mu(u,w))=0$; and therefore,
$\mu(v,\mu(w,v))=0$.
\end{proof}
The following theorem will be useful for this characterization.
\begin{theorem}\label{lemma4.1}
A linear combination of the KV-Poisson structures is a KV-Poisson
structure.
\end{theorem}
\begin{proof}
Let $A$ be an KV-algebra; $\mu_{1}, \mu_{2}$ be two KV-Poisson
structures on $A$. For any $u,v, w$ element of  $A$ and a scalar
$\lambda \neq 0$, we have:
\begin{eqnarray*}
KV_{\mu_{1}+\lambda \mu_{2}}(u, v,w) &=& Ass_{\mu_{1}+\lambda
\mu_{2}}(u,v, w)-Ass_{\mu_{1}+\lambda \mu_{2}}(v,u, w)\\  &=&
KV_{\mu_{1}}(u, v,w) \\ &+&
\lambda[\mu_{2}(\mu_{1}(u,v),w)+\mu_{1}(\mu_{2}(u,v),w)-\mu_{2}(u,\mu_{1}(v,w))-\mu_{1}(u,\mu_{2}(v,w))\\
&-&
\mu_{2}(\mu_{1}(v,u),w)-\mu_{1}(\mu_{2}(v,u),w)+\mu_{2}(v,\mu_{1}(u,w))+\mu_{1}(v,\mu_{2}(u,w))]\\&=&
KV_{\mu_{1}}(u, v,w) + td_{\mu_{1}}\mu_{2}(u, v,w)\\  &=& 0,
\end{eqnarray*}
Therefore, $\mu_{1}+\lambda \mu_{2}$ is a KV-structure. On the other
hand:
\begin{eqnarray*}
(\mu_{1}+\lambda \mu_{2})(u,v)&=& \mu_{1}(u,v)+ \lambda \mu_{2}(u,v)\\
                              &=& -\mu_{1}(v,u)- \lambda \mu_{2}(v,u)\\
                              &=& - (\mu_{1}+\lambda \mu_{2})(v,u),
\end{eqnarray*}
\begin{eqnarray*}
(\mu_{1}+\lambda \mu_{2})(w,(\mu_{1}+\lambda \mu_{2})(u,v))&=& \mu_{1}(w,(\mu_{1}+\lambda \mu_{2})(u,v))\\
                                                           &+& \lambda \mu_{2}(w,(\mu_{1}+\lambda \mu_{2})(u,v))\\
                                                           &=& \mu_{1}(w,\mu_{1}(u,v) )+ \lambda \mu_{1}(w,\mu_{2}(u,v))\\
                                                           &+&  \lambda \mu_{2}(w,\mu_{1}(u,v))+ \lambda^{2} \mu_{2}(w,\mu_{2}(u,v))\\
                                                           &=& \lambda [\mu_{1}(w,\mu_{2}(u,v))+\mu_{2}(w,\mu_{1}(u,v))]\\
                                                           &=& 0
\end{eqnarray*}
Therefore, $\mu_{1}+\lambda \mu_{2}$ satisfied the conditions for a
KV-structure to be a Poisson structure.
\end{proof}
From where the theorem of characterization according to
\begin{theorem}
A KV structure $\mu_P$ on the KV algebra $\mathrm{I\! R}^{2}$ is a
KV-Poisson if it have the following form:
\begin{eqnarray}\label{eq6}
     \mu_{P} =\biggl(\begin{pmatrix}
          0&x_{0}\\
          -x_{0}&0
            \end{pmatrix},
        \begin{pmatrix}
         0&y_{0}\\
         -y_{0}&0
            \end{pmatrix}\biggr)
\end{eqnarray}
for all $x_{0}, y_{0} \in \mathrm{I\! R}$.
\end{theorem}
In other words, we have the following.
\begin{eqnarray}\mu_P=x_0(e_1^*\otimes e_2^*-e_2^*\otimes e_1^*)e_1+y_0(e_1^*\otimes e_2^*-e_2^*\otimes e_1^*)e_2 (x_0e_1+y_0e_2)(e_1^*\otimes e_2^*-e_2^*\otimes e_1^*)\end{eqnarray}
\begin{proof}
In the canonical basis $\{e_{1},e_{2}\}$ of $\mathrm{I\! R}^{2}$, we
have: $\mu_P(e_{i},e_{j})=\sum_{s=1}^{2}C_{ij}^{s}e_{s};$ where
$C_{ij}^{k}$ are the constant structures of $\mu_P$.
\begin{enumerate}
\item[$\bullet$] Skew symmetry. Since $\{e_1, e_2\}$ is free, we
have: for all $i=1,2;$
\begin{eqnarray}
C_{11}^{1}=C_{11}^{2}=C_{22}^{1}=C_{22}^{2}=0.
\end{eqnarray}
if
 $\mu(e_{i},e_{i})=0.$
And when $i\neq j, 1\leq i,j\leq
2$,$\mu_P(e_{i},e_{j})+\mu(e_{j},e_{i})=0$ if and only if
$\sum_{s=1}^{2}(C_{ij}^{s}+C_{ji}^{s})e_{s}=0.$ That is:
$\begin{cases}
C_{ij}^{1}+C_{ji}^{1}=0\\
C_{ij}^{2}+C_{ji}^{2}=0\\
     \end{cases}$\\
Then, while making vary $i$ and $j$ in $\{1, 2\}$, we obtain:
\begin{eqnarray}
\begin{cases}
C_{12}^{1}+C_{21}^{1}=0\\
C_{12}^{2}+C_{21}^{2}=0\\
     \end{cases}
\end{eqnarray}
\item[$\bullet$] Jacobi Identity:\\
We have: $\mu(e_{k},\mu(e_{i},e_{j}))=0$ if and only if
$\sum_{s=1}^{2}(\sum_{l=1}^{2}C_{ij}^{l}C_{kl}^{s})e_{s}=0$, for all
$1\leq i,j, k \leq 2$. That is:
$\sum_{l=1}^{2}C_{ij}^{l}C_{kl}^{s}=0$. And, while making vary $i, j
k$ in $\{1, 2\}$, we obtain:
\begin{eqnarray}
\begin{cases}\label{eq2}
C_{12}^{1}C_{12}^{2}-(C_{12}^{1})^{2}=0\\
(C_{12}^{1})^{2}=0\\
C_{12}^{1}C_{12}^{2}=0\\
     \end{cases}
\end{eqnarray}
Let's put: $x=C_{12}^{1}, y=C_{12}^{2}$; the equation (\ref{eq2})
becomes: $ xy-x^{2}=0, x^{2}=0, xy=0.$ Therefore, $F=F_{1}=\{(x,0),
x \in \mathrm{I\! R}\}\cup F_{2}=\{(0,y), y \in \mathrm{I\! R}\}$ is
set of solutions of (\ref{eq2}) in $\mathrm{I\! R}^{2}$. Set
$F=(\mathrm{I\! R}\times \{0\})\cup (\{0\}\times \mathrm{I\! R}) =
F_{1}\cup F_{2}$; it is an algebraic manifold. So:
\item[$\star$] If $(x_{0},0) \in F_{1}$, we have $C_{12}^{1}=x_{0}, C_{12}^{2}=0$. So:
$\mu_P(e_{i},e_{j})=\sum_{s=1}^{2}C_{ij}^{s}e_{s}=
C_{ij}^{1}e_{1}+C_{ij}^{2}e_{2}$. What is equivalent to:
\begin{eqnarray}
\begin{cases}\label{eq3}
\mu_P(e_{i},e_{i})=0\\
\mu_P(e_{1},e_{2})=C_{12}^{1}e_{1}+C_{12}^{2}e_{2}=x_{0}e_{1}\\
     \end{cases}
\end{eqnarray}
However $\mu_P(u,v)=\mu_{1}(u,v)e_{1}+\mu_{2}(u,v)e_{2}$ if and only
if
$\mu_P(xe_{1}+ye_{2},x'e_{1}+y'e_{2})=\mu_{1}(u,v)e_{1}+\mu_{2}(u,v)e_{2}.$
One has then:\\
$xx'\mu_P(e_{1},e_{1})+xy'\mu_P(e_{1},e_{2})+x'y\mu_P(e_{2},e_{1})+yy'\mu_P(e_{2},e_{2})=\mu_{1}(u,v)e_{1}+\mu_{2}(u,v)e_{2}.$
And therefore:
$(xy'-x'y)\mu_P(e_{1},e_{2})=\mu_{1}(u,v)e_{1}+\mu_{2}(u,v)e_{2}.$
After (\ref{eq3}), one gets:
$x_{0}(xy'-x'y)e_{1}=\mu_{1}(u,v)e_{1}+\mu_{2}(u,v)e_{2}$.
 What implies that:
$\begin{cases}
\mu_{1}(u,v)=x_{0}(xy'-x'y)\\
\mu_{2}(u,v)=0\\
     \end{cases}$\\
 $\mu_{1}$ is therefore an bilinear map with matrix
    $ \Gamma_{1} =\begin{pmatrix}
          0&x_{0}\\
          -x_{0}&0
            \end{pmatrix}$\\
In the same way,
\item[$\bullet$] If $(0,y_{0}) \in F_{2}$, one has
$C_{12}^{1}=0, C_{12}^{2}=y_{0}$, we have therefore:
$\mu(e_{i},e_{j})=\sum_{s=1}^{2}C_{ij}^{s}e_{s}=
C_{ij}^{1}e_{1}+C_{ij}^{2}e_{2}.$ We have therefore: $\begin{cases}
\mu(e_{i},e_{i})=0\\
\mu(e_{1},e_{2})=C_{12}^{1}e_{1}+C_{12}^{2}e_{2}=y_{0}e_{2}\\
     \end{cases}$\\
And therefore:
$y_{0}(xy'-x'y)e_{2}=\mu_{1}(u,v)e_{1}+\mu_{2}(u,v)e_{2}$.
 Then:
$\begin{cases}
\mu_{1}(u,v)=0\\
\mu_{2}(u,v)=y_{0}(xy'-x'y)\\
     \end{cases}$\\
 $\mu_{2}$ is an bilinear map with matrix
     $\Gamma_{2} =\begin{pmatrix}
          0&y_{0}\\
          -y_{0}&0
            \end{pmatrix}$\\
And according to the lemma \ref{lemma4.1}, the sub vectorial space
generated by $F$ is generated by
\begin{eqnarray*}
     \mu_{P} =\biggl(\begin{pmatrix}
         0&x_{0}\\
          -x_{0}&0
            \end{pmatrix},
        \begin{pmatrix}
       0&y_{0}\\
          -y_{0}&0
            \end{pmatrix}\biggr)                                    
\end{eqnarray*}
\end{enumerate}
This complete the proof.
\end{proof}
In the follows, we determine the KV-Poisson structures in the case
where the KV structure is a non-degenerate symmetric structure. We
recall that $Sol((\mathrm{I\! R}^{2}, KV)$ denotes the set of KV
structures on $\mathrm{I\! R}^2$; see \cite{Mopeng}.
\begin{proposition}
Let $\mu$ be KV-structure of the KV-algebra $\mathrm{I\! R}^{2}$ and
$\mu_{P}$ associate. KV-Poisson structure, then, we have:
\begin{enumerate}
\item[(i)] If $\mu$ is symmetric and non degenerated
 (hessian), then
\begin{center}
 $ \mu_{P} =\biggl(\begin{pmatrix}
          0&0\\
          0&0
            \end{pmatrix},
        \begin{pmatrix}
         0&0\\
         0&0
            \end{pmatrix}\biggr),$
\end{center}
\item[(ii)] If $\mu$ is not symmetric and not degenerate, then
\begin{center}
   $  \mu_{P} =\biggl(\begin{pmatrix}
          0&a\\
          -a&0
            \end{pmatrix},
        \begin{pmatrix}
         0&c\\
         -c&0
            \end{pmatrix}\biggr).$
\end{center}
\end{enumerate}
\end{proposition}
\begin{proof}
\begin{enumerate}
\item[(i)] Let $\mu \in Sol((\mathrm{I\! R}^{2}, KV)$ such that:
    $ \mu =\biggl(\begin{pmatrix}
          a&0\\
          0&b
            \end{pmatrix},
        \begin{pmatrix}
         c&0\\
         0&d
            \end{pmatrix}\biggr).$\\
with $a\neq 0,b\neq 0,c\neq 0,d\neq 0 .$ For all $u=(x,y),
v=(x',y')$, we have:
    $ \mu(u,v)= (axx'+byy', cxx'+dyy')$ if and only if $\mu(xe_{1}+ye_{2},x'e_{1}+y'e_{2})=(axx'+byy')e_{1}+(cxx'+dyy')e_{2}.$
What is equivalent to:\\
$xx'\mu(e_{1},e_{1})+xy'\mu(e_{1},e_{2})+yx'\mu(e_{2},e_{1})+yy'\mu(e_{2},e_{2}) =xx'(ae_{1}+ce_{2})+yy'(be_{1}+de_{2})$\\
What implies therefore that: $\begin{cases}
\mu(e_{1},e_{1})= ae_{1}+ce_{2}=0\\
\mu(e_{1},e_{2})=0\\
\mu(e_{2},e_{2})=be_{1}+de_{2}=0\\
\end{cases}$\\
So therefore: $a=b=c=d=0.$
\item[(ii)] Let $\mu \in
Sol((\mathrm{I\! R}^{2}, KV)$ such that:
     $\mu =\biggl(\begin{pmatrix}
          0&a\\
          b&0
            \end{pmatrix},
        \begin{pmatrix}
         0&c\\
         d&0
            \end{pmatrix}\biggr),$\\
with $a\neq b, c\neq d,a\neq 0,b\neq 0,c\neq 0,d\neq 0.$
For all $u=(x,y), v=(x',y')$, we have:\\
$ \mu(u,v) =(axy'+byx', cxy'+dyx')$ if and only if
$\mu(xe_{1}+ye_{2},x'e_{1}+y'e_{2}) =(axy'+byx', cxy'+dyx').$
What is equivalent therefore to:\\
$xx'\mu(e_{1},e_{1})+xy'\mu(e_{1},e_{2})+yx'\mu(e_{2},e_{1})+yy'\mu(e_{2},e_{2})
=xy'(ae_{1}+ce_{2})+yx'(be_{1}+de_{2}).$ One has therefore:
$\begin{cases}
\mu(e_{1},e_{1})= 0\\
\mu(e_{1},e_{2})=ae_{1}+ce_{2}\\
\mu(e_{2},e_{1})=be_{1}+de_{2}\\
\mu(e_{2},e_{2})=0\\
\end{cases}$\\
What implies that: $b=-a, d=-c.$ On the other hand: $ \mu(u,v)
=\mu_{1}(u,v)e_{1}+\mu_{2}(u,v)e_{2}$; While using the fact that
$\mu$ is skew symmetric, we have: $(xy'-yx')\mu(e_{1},e_{2})
=\mu_{1}(u,v)e_{1}+\mu_{2}(u,v)e_{2}.$ What implies therefore that:
$(xy'-yx')(ae_{1}+ce_{2}) =\mu_{1}(u,v)e_{1}+\mu_{2}(u,v)e_{2}.$
That is $(xy'-yx')ae_{1}+(xy'-yx')ce_{2})
=\mu_{1}(u,v)e_{1}+\mu_{2}(u,v)e_{2}.$ Therefore: $\begin{cases}
\mu_{1}(u,v)=a(xy'-yx')\\
\mu_{2}(u,v)=c(xy'-yx')\\
\end{cases}$\\
$\mu_{1}(u,v)$ is an associative, bilinear form with matrix
     $\Gamma_{1} =\begin{pmatrix}
          0&a\\
          -a&0
            \end{pmatrix}$\\
and $\mu_{2}(u,v)$, is an associative, bilinear form with matrix
    $ \Gamma_{2} =\begin{pmatrix}
          0&c\\
          -c&0
            \end{pmatrix}.$\\
From where
     $\mu_{P} =\biggl(\begin{pmatrix}
          0&a\\
          -a&0
            \end{pmatrix},
        \begin{pmatrix}
         0&c\\
         -c&0
            \end{pmatrix}\biggr)$
\end{enumerate}
\end{proof}
\section{Complex of KV-Poisson Cohomology.}\label{sec4}
In this part, we study the homological invariants associated to
these structures of KV-Poisson structure in order to study their
deformation.
\subsection{Generalities}
In this subsection we define the notion of KV-Poisson Cohomology of
KV-Poisson structure. Let $(A, [,])$ be a final dimensional real Lie
algebra of dimension $n;$ and $M$ be a final dimensional
$\mathrm{I\! R}$ vector space. A representation of $A$ is any
morphism of Lie algebras $\rho \colon A \rightarrow gl( M)$; such
that: for all $a,b \in A$, $\rho([a,b])=
\rho(a)\circ\rho(b)-\rho(b)\circ\rho(a)$. Obviously, all
representations lead to a structure of $A$-module on $M$ defined by
$(a, m)\mapsto \rho(a)m.$ Let $q$ be a positive integer; a $q$ form
on $A$ with values in $M$ is any skew symmetric $\mathrm{I\!
K}-$multilinear map :$f \colon A\times A \times \dots \times A
\rightarrow M.$ One will note by $C^{q}(A,M)$ the set of $q$-forms.
Consider the graded vector space $C(A,M)=\oplus_{q}C^{q}(A,M),$
where $\begin{cases}
C^{q}(A,M)=0, q<0\\
C^{0}(A,M)=M\\
C^{q}(A,M)=Hom(\Lambda^{q}A,M), q \geq 1\\
\end{cases}$\\
One defines the following  coboundary: $\delta^{q} \colon C^{q}(A,M)
\rightarrow C^{q+1}(A,M),$ by:
\begin{eqnarray*}
(\delta^{q} f)(a_{0} ...a_{q})&=&\Sigma_{i}(-1)^{i}\rho(a_{i}) (f(a_{0},...\widehat{a_{i}}...a_{q})\\
                              &+&\Sigma_{i<j}(-1)^{i+j}f([a_{i},a_{j}],...,\widehat{a_{i}} , ..., \widehat{a_{j}},...,a_{q})
\end{eqnarray*}
where for $q=0,1$ one have:
\begin{flushleft}
$\bullet      (\delta^{0}\xi)(a) = a\xi, \xi \in M $\\
$\bullet (\delta^{1}f_{1})(a, b)=\rho(a).f_{1}(b)-\rho(b).f_{1}(a)-f_{1}([a,b])$\\
\end{flushleft}
In the case where $M=A$, and $(A, \mu)$ is a Lie algebra, the map
\begin{eqnarray*}
ad: a \mapsto ad_{x}\in End_{\mathrm{I\! K}}( A), ad_{x}:y \mapsto
ad_{x}(y)=\mu(a,b)\in A, \forall y,a, x \in A
\end{eqnarray*} is called adjoint of $\mu$. In this particular case, the Jacobi identity allows us to show that $\rho=ad$
 is a representation of $A$ by $gl(A)$. This is usually called an adjoint representation. In this case the differential operator is defined by

\begin{eqnarray}\label{11-13}
(\delta^{0}_\mu\xi)(a)&=& a\xi, \xi \in A\\
(\delta^{q}_\mu f)(a_{0},...,a_{q})&=&\underset{i=1}{\overset{q}\sum}(-1)^{i}a_{i}(f(a_{0},...\widehat{a_{i}}...a_{q})\\
                              &+&\underset{i<j}{\sum}(-1)^{i+j}f(\mu(a_{i},a_{j}),...,\widehat{a_{i}} , ...,
                              \widehat{a_{j}},...,a_{q})\nonumber
\end{eqnarray} for all $f\in C^q(A):=C^q(A, A)$
this map satisfy the relation $\delta^{q+1}\circ\delta^{q}=0.$
\begin{definition}\cite{BOYOM4}
The complex $\biggl(C^{*}(A,M),\delta^{*}\biggr)$ is called
Chevalley-Eilenberg complex.
\end{definition}
\subsection{\textbf{KV-Poisson Cohomology of $\mathrm{I\! R}^{2}$}}
In this part we define and characterize the first three Cohomology
groups associated with a KV-Poisson structure $\mu_{P}$ given by
(\ref{eq6}), where $(\mathrm{I\! R}^{2},\mu_{p})$ is an algebra of
KV-Poisson associated. We have:
\begin{center}
  $\mu_{p} =\biggl(\begin{pmatrix}
          0&x_{0}\\
          -x_{0}&0
            \end{pmatrix},
        \begin{pmatrix}
         0&y_{0}\\
         -y_{0}&0
            \end{pmatrix}\biggr)$.
\end{center}
For all $u=(x,y), v=(x^{'},y^{'})$ elements for $\mathrm{I\!
R}^{2}$, we have:
\begin{equation}
\mu_{p}(u,v)= (x_{0}xy^{'}-x_{0}yx^{'}, y_{0}xy^{'}-y_{0}yx^{'}).
\end{equation}
We denote by $\chi^{q}(\mathrm{I\! R}^{2})$ the space of multilinear alternating applications of $\mathrm{I\! R}^{2}$.\\
Consider the complex:
\begin{eqnarray}
 \chi^{0}(\mathrm{I\! R}^{2})\stackrel{ \delta^{0} }{\longrightarrow} \chi^{1}(\mathrm{I\! R}^{2})\stackrel{ \delta^{1} }{\longrightarrow}\chi^{2}(\mathrm{I\! R}^{2})\stackrel{ \delta^{2} }{\longrightarrow}0.
\end{eqnarray}
with the KV-Poisson cobord operator:
\begin{eqnarray*}
(\delta^{0}\xi)(a)&=& \mu_{p}(a,\xi), \xi \in \mathrm{I\! R}^{2}\\
(\delta^{q} f)(a_{0} ...a_{q})&=&\Sigma_{i}(-1)^{i}\mu_{P}[a_{i},(f(a_{0},...\widehat{a_{i}}...a_{q}))]\\
                              &+&\Sigma_{i<j}(-1)^{i+j}f(\mu_{P}(a_{i},a_{j}),...,\widehat{a_{i}} , ..., \widehat{a_{j}},...,a_{q})
\end{eqnarray*}
We have the following Theorem
\begin{theorem}Let $(\mathrm{I\! R}^{2},\mu_{p})$ is an algebra of
KV-Poisson associated, with $\mu_{p}$ is a KV-Poisson structure
given by
\begin{center}
  $\mu_{p} =\biggl(\begin{pmatrix}
          0&x_{0}\\
          -x_{0}&0
            \end{pmatrix},
        \begin{pmatrix}
         0&y_{0}\\
         -y_{0}&0
            \end{pmatrix}\biggr)$.
\end{center} The first  Cohomology
groups associated with a KV-Poisson structure $\mu_{P}$ is
characterized  by
\[
H_{P}^{0}(\mu_{P})\simeq  \begin{cases}
\{0\},&\text{if $x_{0}\neq 0$}\\
\mathrm{I\! R}^{2}, &\text{if $x_{0}=y_{0}=0$}
     \end{cases}
\]
\end{theorem}
\begin{proof}
For all $u=(x,y)$ element for $\mathrm{I\! R}^{2}$,
$\xi=(\xi_{1},\xi_{2}) \in \mathrm{I\! R}^{2}$, we have:
\begin{eqnarray*}
(\delta^{0}\xi)(u)&=& u\xi\\
                  &=& \mu_{P}(u,\xi)\\
                  &=& (u^{t}\Gamma_{1}\xi, u^{t}\Gamma_{2}\xi)\\
                  &=& (-x_{0}y\xi_{1}+x_{0}x\xi_{2}, -y_{0}y\xi_{1}+y_{0}x\xi_{2})
\end{eqnarray*}
So
\begin{eqnarray*}
(\delta^{0}\xi)(u)=0 & \Longleftrightarrow & (\delta^{0}\xi)(e_{i})=0, 1 \leq i, \leq 2\\
& \Longleftrightarrow & \begin{cases}
-x_{0}\xi_{1}=0\\
x_{0}\xi_{2}=0\\
-y_{0}\xi_{1}=0\\
y_{0}\xi_{2}=0\\
     \end{cases}
\end{eqnarray*}
On the other hand, $Im\delta^{-1}=\{0\}.$\\
$\star$ If $x_{0}\neq 0, y_{0}\neq 0$, we have: $\xi_{1}=\xi_{2}=0$;
$\xi=(0,0)$. And so $Ker\delta^{0} = \{0\}$. Hence:
\begin{center}
  $H_{P}^{0}(\mu_{P})\simeq \{0\}$.
\end{center}
$\star$ If $x_{0}\neq 0, y_{0}= 0$, we have: $\xi_{1}=\xi_{2}=0$;
$\xi=(0,0)$. And so $Ker\delta^{0} = \{0\}$. Hence:
\begin{center}
  $H_{P}^{0}(\mu_{P})\simeq \{0\}$.
\end{center}
$\star$ If $x_{0}= 0, y_{0}= 0$, we have: $\xi=(\xi_{1},\xi_{2})$.
And so $Ker\delta^{0} = Vect \biggl\{(1,0),(0,1) \biggl\}= Vect
\biggl\{e_{1},e_{2}\biggl\}$. Hence:
\begin{center}
  $H_{P}^{0}(\mu_{P})\simeq \mathrm{I\! R}^{2}$.
\end{center}
\end{proof}
We have the Theorem
\begin{theorem}Let $(\mathrm{I\! R}^{2},\mu_{p})$ is an algebra of
KV-Poisson associated, with $\mu_{p}$ is a KV-Poisson structure
given by
\begin{center}
  $\mu_{p} =\biggl(\begin{pmatrix}
          0&x_{0}\\
          -x_{0}&0
            \end{pmatrix},
        \begin{pmatrix}
         0&y_{0}\\
         -y_{0}&0
            \end{pmatrix}\biggr)$.
\end{center} Let  $\chi^{1}(\mathrm{I\! R}^{2})$ the $2^{nd}$ space of multilinear alternating applications of $\mathrm{I\! R}^{2}$. The second  Cohomology
groups associated with a KV-Poisson structure $\mu_{P}$ is
characterized  by
\[
H_{P}^{1}(\mu_{P})\simeq  \begin{cases}
\{0\},&\text{if $x_{0}\neq 0$}\\
\chi^{1}(\mathrm{I\! R}^{2}), &\text{if $ x_{0}=y_{0}=0$}
     \end{cases}
\]
\end{theorem}
\begin{proof}
For all $u=(x,y), v=(x',y')$ elements for $\mathrm{I\! R}^{2}$, we
have:
\begin{eqnarray*}
(\delta^{1}f_{1})(u,
v)=\mu_{P}(u,f_{1}(v))-\mu_{P}(v,f_{1}(u))-f_{1}(\mu_{P}(u,v))
\end{eqnarray*}
where $f_{1}$ is a linear application of $\mathrm{I\! R}^{2}$ in
$\mathrm{I\! R}^{2}$; therefore associated with the matrix
    $ A =\begin{pmatrix}
          \alpha& \beta\\
    \gamma&\lambda
            \end{pmatrix}$,
in a base $(e_{1}, e_{2})$ for $\mathrm{I\! R}^{2}$.\\
And so: $f_{1}(u)=(\alpha x+\beta y,\gamma x+\lambda y)$.
\begin{eqnarray*}
\bullet  \mu_{P}(u,f_{1}(v)) &=&(u^{t}\Gamma_{1}f_{1}(v),u^{t}\Gamma_{2}f_{1}(v))  \\
                   &=& (x_{0}\gamma xx'+x_{0}\lambda xy' -x_{0}\alpha yx'-x_{0}\beta yy',y_{0}\gamma xx'\\
                   &+& y_{0}\lambda xy' -y_{0}\alpha yx'-y_{0}\beta yy')
\end{eqnarray*}
\begin{eqnarray*}
\bullet  \mu_{P}(v,f_{1}(u)) &=&(v^{t}\Gamma_{1}f_{1}(u),v^{t}\Gamma_{2}f_{1}(u))  \\
                   &=& (x_{0}\gamma xx'+x_{0}\lambda yx' -x_{0}\alpha xy'-x_{0}\beta yy', y_{0}\gamma xx'\\
                   &+& y_{0}\lambda yx' -y_{0}\alpha xy'-y_{0}\beta yy')
\end{eqnarray*}
\begin{eqnarray*}
\bullet  f_{1}(\mu_{P}(u,v)) =\biggl((\alpha x_{0}+\beta y_{0})xy'-(\alpha x_{0}+\beta y_{0})yx',\\
                              (\gamma x_{0}+\lambda y_{0})xy'- (\gamma x_{0}+\lambda y_{0})yx'\biggr)
\end{eqnarray*}
Hence the expression $(\delta^{1}f_{1})(u, v)$:
\begin{eqnarray*}
(\delta^{1}f_{1})(u, v)&=&\biggl((\lambda x_{0}-\beta
y_{0})(xy'-yx'),(\alpha y_{0}-\gamma x_{0})(xy'-yx')\biggr)
\end{eqnarray*}
by remarking that
  $(\delta^{1}f_{1})(u, u)=(0,0)$. This implies that $(\delta^{1}f_{1})$ is alternating; therefore, is an antisymmetric bilinear form of $\mathrm{I\! R}^{2}$.
But
\begin{eqnarray*}
\delta^{1}f_{1}(u,v)=0 & \Longleftrightarrow & \delta^{1}f_{1}(e_{i},e_{j})=0, 1 \leq i, j \leq 2\\
& \Longleftrightarrow & \begin{cases}
\lambda x_{0}-\beta y_{0}=0\\
-\lambda x_{0}+\beta y_{0}=0\\
\alpha y_{0}-\gamma x_{0}=0\\
-\alpha y_{0}+\gamma x_{0}=0\\
     \end{cases}
\end{eqnarray*}
$\star$ If $x_{0}\neq 0, y_{0}\neq 0$, we have: $ \alpha=\beta = \lambda = \gamma = 0$. We therefore have $f_{1} =0$.\\
$\star$ If $x_{0}\neq 0, y_{0}= 0$, we have: $ \lambda= \gamma= 0$;
and therefore: $f_{1} \simeq \begin{pmatrix}
          \alpha& \beta\\
           0& 0
            \end{pmatrix}$. We have
\begin{center}
  $Ker\delta^{1}= Vect \biggl\{\begin{pmatrix}
          1&0\\
          0&0
            \end{pmatrix},
            \begin{pmatrix}
          0&1\\
           0& 0
            \end{pmatrix}\biggr\}=Vect \biggl\{E_{11}, E_{12}\biggr\}$.
\end{center}
$\star$ If $x_{0}= 0, y_{0}= 0$, we have: $f_{1} \simeq
\begin{pmatrix}
          \alpha& \beta\\
          \gamma& \lambda
            \end{pmatrix}$, and therefore:
\begin{eqnarray*}
Ker\delta^{1}&=& Vect \biggl\{\begin{pmatrix}
          1&0\\
          0&0
            \end{pmatrix},
            \begin{pmatrix}
          0&1\\
           0& 0
            \end{pmatrix},
           \begin{pmatrix}
          0&0\\
           1& 0
            \end{pmatrix},
           \begin{pmatrix}
          0&0\\
           0& 1
            \end{pmatrix}\biggr\}\\
           &=& Vect \biggl\{E_{11}, E_{12}, E_{21},E_{22}\biggr\}
\end{eqnarray*}
So, consider $g$ be a linear application of $\mathrm{I\! R}^{2}$
$\longrightarrow$ $\mathrm{I\! R}^{2}$, associated matrix
     $U= \begin{pmatrix}
          u_{11}& u_{12}\\
           u_{21}& u_{22}
            \end{pmatrix}, $
in a base $(e_{1}, e_{2})$ for $\mathrm{I\! R}^{2}$.\\
For all  $u=(x,y)$ element for $\mathrm{I\! R}^{2}$, we have
$g(u)=(u_{11}x+u_{12}y,u_{21}x+u_{22}y).$
\begin{eqnarray*}
g \in Im\delta^{0} &\Leftrightarrow& (\delta^{0}\xi)(u)=g(u)\\
                   &\Leftrightarrow& (\delta^{0}\xi)(e_{i})=g(e_{i}), 1\leq i \leq 2\\
                   &\Leftrightarrow& \begin{cases}
x_{0}\xi_{2}=u_{11}\\
-x_{0}\xi_{1}=u_{12}\\
y_{0}\xi_{2}=u_{21}\\
-y_{0}\xi_{1}=u_{22}\\
\end{cases}
\end{eqnarray*}
$\star$ If $x_{0}\neq 0, y_{0}\neq 0$, we have: $g= \begin{pmatrix}
          u_{11}& u_{12}\\
           u_{21}& u_{22}
            \end{pmatrix}.$\\
And $g$ is an element of $Ker\delta^{1}$. we have
\begin{eqnarray*}
(\delta^{1}g)(u,v)=0 &\Leftrightarrow& (\delta^{1}g)(e_{i},e_{j})=0,  1 \leq i,j\leq 2\\
                   &\Leftrightarrow& \begin{cases}
x_{0}u_{22}-y_{0}u_{12}=0\\
-x_{0}u_{22}+y_{0}u_{12}=0\\
y_{0}u_{11}-x_{0}u_{21}=0\\
-y_{0}u_{11}+x_{0}u_{21}=0\\
\end{cases}
\end{eqnarray*}
We deduces that: $u_{22}=u_{12}=u_{11}=u_{21}=0$. So: $g=0$. Hence
\begin{center}
 $H_{P}^{1}(\mu_{P})= \{0\}$
\end{center}
$\star$ If $x_{0}\neq 0, y_{0}=0$, we have: $u_{21}= u_{22}=0$ . So
$g= \begin{pmatrix}
          u_{11}& u_{12}\\
           0& 0
            \end{pmatrix} \in Ker\delta^{1}.$\\
So, $Im\delta^{0}= Vect \biggl\{E_{11}, E_{12}\biggr\}$. Hence
\begin{center}
  $H_{P}^{1}(\mu_{P})= \{0\}$
\end{center}
$\star$ If $x_{0}= y_{0}=0$, we have: $u_{11}= u_{12}=u_{21}=
u_{22}=0$; and so $Im\delta^{0}=\{0\}$. Hence
\begin{eqnarray*}
H_{P}^{1}(\mu_{P})&=& Vect \biggl\{E_{11}, E_{12}, E_{21},E_{22}\biggr\}\\
                  &=&\chi^{1}(\mathrm{I\! R}^{2})
\end{eqnarray*}
\end{proof} We have the following theorem
\begin{theorem}Let $(\mathrm{I\! R}^{2},\mu_{p})$ is an algebra of
KV-Poisson associated, with $\mu_{p}$ is a KV-Poisson structure
given by
\begin{center}
  $\mu_{p} =\biggl(\begin{pmatrix}
          0&x_{0}\\
          -x_{0}&0
            \end{pmatrix},
        \begin{pmatrix}
         0&y_{0}\\
         -y_{0}&0
            \end{pmatrix}\biggr)$.
\end{center} Let  $\chi^{2}(\mathrm{I\! R}^{2})$ the third space of multilinear alternating applications of $\mathrm{I\! R}^{2}$. The third Cohomology
groups associated with a KV-Poisson structure $\mu_{P}$ is
characterized  by
\[
H_{P}^{2}(\mu_{P})\simeq  \begin{cases}
\{0\},&\text{if $x_{0}\neq 0$}\\
\chi^{2}(\mathrm{I\! R}^{2}), &\text{ if $x_{0}=y_{0}=0$}
     \end{cases}
\]
\end{theorem}
\begin{proof}
Let  $f_{2} \in \chi^{2}(\mathrm{I\! R}^{2})$  such that:\\
$f_{2}: \mathrm{I\! R}^{2}\times \mathrm{I\! R}^{2} \longrightarrow
\mathrm{I\! R}^{2}$,is an antisymmetric bilinear application of
matrices $ C =\begin{pmatrix}
          0& f\\
          -f&0
            \end{pmatrix},
 D =\begin{pmatrix}
          0&j\\
          -j&0
            \end{pmatrix}, $
in a base $(e_{1}, e_{2})$ for $\mathrm{I\! R}^{2}$.\\
For all $u=(x,y), v=(x',y'), w=(x'',y'')$ elements for $\mathrm{I\!
R}^{2}$, we have:
\begin{center}
  $f_{2}(u,v)=(fxy'-fyx', jxy'-jyx')$.
\end{center}
We have $f_{2}(u,u)=(0, 0)$; and so $f_{2}$ is antisymmetric. And
\begin{eqnarray*}
(\delta^{2}f_{2})(u, v,w)&=& \mu_{P}(u,f_{2}(v,w))-\mu_{P}(v,f_{2}(u,w))+\mu_{P}(w,f_{2}(u,v))-f_{2}(\mu_{P}(u,v),w)\\
                          &-& f_{2}(\mu_{P}(v,w),u)+f_{2}(\mu_{P}(u,w),v);
\end{eqnarray*}
where
\begin{eqnarray*}
\bullet  \mu_{P}(u,f_{2}(v,w)) &=& (u^{t}\Gamma_{1}f_{2}(v,w),u^{t}\Gamma_{2}f_{2}(v,w))  \\
                               &=& \biggl(jx_{0} xx'y'' - jx_{0}xy'x''- fx_{0}yx'y''\\
                               &+& f x_{0} yy'x'' ,jy_{0}xx'y'' -jy_{0} xy'x'' -fy_{0}yx'y''+ fy_{0}yy'x'' \biggr)
\end{eqnarray*}
\begin{eqnarray*}
\bullet  \mu_{P}(v,f_{2}(u,w)) &=& (v^{t}\Gamma_{1}f_{2}(u,w),v^{t}\Gamma_{2}f_{2}(u,w))  \\
                   &=& \biggl(jx_{0}xx'y'' - jx_{0}yx'x''- fx_{0}xy'y''\\
                   &+& fx_{0}yy'x'' , jy_{0}xx'y'' -jy_{0} yx'x''-fy_{0}xy'y''+ fy_{0}yy'x'' \biggr)
\end{eqnarray*}
\begin{eqnarray*}
\bullet  \mu_{P}(w,f_{2}(u,v)) &=& (w^{t}\Gamma_{1}f_{2}(u,v),w^{t}\Gamma_{2}f_{2}(u,v))  \\
                   &=& \biggl(jx_{0}xy'x'' -jx_{0}yx'x''- fx_{0}xy'y''\\
                   &+& fx_{0}yx'y'' ,jy_{0}xy'x'' -jy_{0}yx'x''-fy_{0}xy'y''+ fy_{0}yx'y'' \biggr)
\end{eqnarray*}
\begin{eqnarray*}
\bullet  f_{2}(\mu_{P}(u,v),w) &=& (\mu_{P}(u,v)^{t}Cw,\mu_{P}(u,v)^{t}Dw)  \\
                   &=& \biggl( fx_{0}xy'y''-fx_{0}yx'y''-fy_{0}xy'x''\\
                   &+& fy_{0}yx'x'', jx_{0}xy'y''-jx_{0}yx'y''-jx_{0}xy'x''+ jy_{0}yx'x''\biggr)
\end{eqnarray*}
\begin{eqnarray*}
\bullet  f_{2}(\mu_{P}(v,w),u) &=& (\mu_{P}(v,w)^{t}Cu,\mu_{P}(v,w)^{t}Du)  \\
                   &=& \biggl(fx_{0}yx'y''-fx_{0}yy'x''-fy_{0}xx'y''\\
                   &+& fy_{0}xy'x'', jx_{0}yx'y''-jx_{0}yy'x''- jy_{0}xx'y''+ jy_{0}xy'x''\biggr)
\end{eqnarray*}
\begin{eqnarray*}
\bullet  f_{2}(\mu_{P}(u,w),v) &=& (\mu_{P}(u,w)^{t}Cv,\mu_{P}(u,w)^{t}Dv)  \\
                   &=& \biggl(fx_{0}xy'y''-fx_{0}yy'x''-fy_{0}xx'y''\\
                   &+& fy_{0}yx'x'', jx_{0}xy'y''-jx_{0}yy'x''- jy_{0}xx'y''+ jy_{0}yx'x''\biggr)
\end{eqnarray*}
Hence
\begin{eqnarray}
(\delta^{2}f_{2})(u, v,w) &=& \biggl(0,0\biggr), \forall f_{2} \in
\chi^{2}\mathrm{I\! R}^{2}
\end{eqnarray}
We have $Ker\delta^{2}=\chi^{2}(\mathrm{I\! R}^{2})$.\\
So, we have let $g$ be an antisymmetric bilinear application of
$\mathrm{I\! R}^{2} \times \mathrm{I\! R}^{2}$ $\longrightarrow$
$\mathrm{I\! R}^{2}$, of associated matrices:
    $ U= \begin{pmatrix}
          0& u_{12}\\
           -u_{12}& 0
            \end{pmatrix},
       V= \begin{pmatrix}
          0& v_{12}\\
           -v_{12}&0
            \end{pmatrix}, $
in a base $(e_{1}, e_{2})$ for $\mathrm{I\! R}^{2}$.\\
For all $u=(x,y),v=(x',y')$ elements for $\mathrm{I\! R}^{2}$, we
have:
\begin{eqnarray*}
g(u,v)=( u_{12}xy'-u_{12}yx', v_{12}xy'-v_{12}yx')
\end{eqnarray*}
\begin{eqnarray*}
g\in Im \delta^{1} &\Longleftrightarrow & \delta^{1}f_{1}(u,v)=g(u,v)\\
                   &\Longleftrightarrow &  \delta^{1}f_{1}(e_{i},e_{j})=g(e_{i},e_{j}), 1\leq i, j \leq 2\\
                   &\Longleftrightarrow & \begin{cases}
\lambda x_{0}-\beta y_{0}=u_{12}\\
\alpha y_{0}-\gamma x_{0}=v_{12}\\
\end{cases}
\end{eqnarray*}
$\star$ If $x_{0}\neq 0, y_{0}\neq 0$, we have: $(\alpha \lambda-\beta \gamma). A^{-1}X=B$; where $X=(x_{0},y_{0}), B=(u_{12},v_{12})$.\\
And so, $(x_{0},y_{0})$ is an solution , for all $(u_{12},v_{12})$.\\
 $g = \biggl( \begin{pmatrix}
          0& u_{12}\\
           -u_{12}& 0
            \end{pmatrix},
        \begin{pmatrix}
         0& v_{12}\\
           -v_{12}& 0
            \end{pmatrix}\biggr) \in Ker\delta^{2}$ . We have $Im\delta^{1}=\chi^{2}(\mathrm{I\! R}^{2})$. Hence:
 \begin{center}
   $H_{P}^{2}(\mu_{P})\simeq \{0\}$
 \end{center}
$\star$ If $x_{0}\neq 0, y_{0}= 0$, we also have: $\begin{cases}
\lambda x_{0}=u_{12}\\
-\gamma x_{0}=v_{12}\\
\end{cases}$.\\
$g = \biggl( \begin{pmatrix}
          0& u_{12}\\
           -u_{12}& 0
            \end{pmatrix},
        \begin{pmatrix}
         0& v_{12}\\
           -v_{12}& 0
            \end{pmatrix}\biggr) \in Ker\delta^{2}$ . We have $Im\delta^{1}=\chi^{2}(\mathrm{I\! R}^{2})$. Hence:
 \begin{center}
   $H_{P}^{2}(\mu_{P})\simeq \{0\}$
 \end{center}
$\star$ If $x_{0}= y_{0}= 0$, we have : $u_{12}=v_{12}= 0$.
$Im\delta^{1}=\{0\}$. Hence:
 \begin{center}
   $H_{P}^{2}(\mu_{P})\simeq \chi^{2}(\mathbb{R}^{2})$.
 \end{center}
\end{proof}

\section{General conclusion}\label{sec5}
The classification of the KV-Poisson structures on $ \mathrm{I\!
R}^{2}$ offer a deepened survey and detailed of the different
structures of KV-Poisson on $ \mathrm{I\! R}^{2} $.
 The fundamental concepts, the
geometric and algebraic properties, as well as the mathematics
methods  used to characterize these structures have presented
summers

\backmatter

\bmhead{Supplementary information} This manuscript has no additional
data.

\bmhead{Acknowledgments} We would like to thank all the active
members of the algebra and geometry research group at the University
of Maroua in Cameroon.

\section*{Declarations}
This article has no conflict of interest to the journal. No
financing with a third party.
\begin{itemize}
\item No Funding
\item No Conflict of interest/Competing interests (check journal-specific guidelines for which heading to use)
\item  Ethics approval
\item  Consent to participate
\item  Consent for publication
\item  Availability of data and materials
\item  Code availability
\item Authors' contributions
\end{itemize}
\bibliography{bibliography}

\end{document}